\documentclass[11pt]{amsart}
\usepackage{graphicx}
\usepackage{amssymb}
\usepackage{amsthm}
\usepackage{amsmath}
\usepackage{epsfig}
\usepackage{natbib}
\usepackage[bookmarks=open,colorlinks=true, allcolors=blue]{hyperref}
\usepackage{hyphenat}
\usepackage{enumitem}
\usepackage{xcolor}
\newtheorem{theorem}{Theorem}
\newtheorem{theorem*}{Theorem}
\newtheorem{lemma}{Lemma}
\newtheorem{prop}{Proposition}
\newtheorem{cor}{Corollary}

\theoremstyle{definition}
\newtheorem{example}{Example}
\newtheorem{defi}{Definition}[section]
\newtheorem{remark}{Remark}

\begin{document}

\title{Partitioned Factors in Christoffel and Sturmian Words}

\author{Norman Carey \and David Clampitt }
\email{\textsf{\href{mailto:ncarey@gc.cuny.edu}{N. Carey $<$ncarey@gc.cuny.edu$>$},
      \href{clampitt.4@osu.edu}{D. Clampitt $<$clampitt.4@osu.edu$>$}}.} 
\maketitle

\begin{abstract}
Borel and Reutenauer (2006) showed, \emph{inter alia}, that a word $w$ of length $n>1$ is conjugate to a Christoffel word  if and only if for $k=0,1, \dots , n-1$, $w$ has $k+1$ distinct circular factors of length $k$. Sturmian words are the infinite counterparts to Christoffel words, characterized as aperiodic but of minimal complexity, i.e., for all $n \in \mathbb{N}$ there are $n+1$ factors of length $n$. Berth\'{e} (1996) showed that the factors of a given length in the Sturmian case have at most three frequencies (probabilities). In this paper we extend to results on factors of both Christoffel words and Sturmian words under fixed partitionings (decompositions of factors of length $m$ into concatenations of words whose lengths are given by a composition of $m$ into $k$ components). Any factor of a Sturmian word (respectively, circular factor of a Christoffel word) thus partitioned into $k$ elements belongs to one of $k+1$ equivalence classes (varieties). We show how to compute the sizes of the equivalence classes determined by a composition in the case of Christoffel word, and show how to compute the frequencies of the classes in the case of Sturmian words. A version of the finitary case was proved in a very different context (expressed in musical and number-theoretical terms) by Clough and Myerson (1985, 1986); we use their terminology, \emph{variety} 
and \emph{multiplicity}.
\end{abstract}


\section{Partitioned Factors in Christoffel Words}\label{incw}

\subsection{Introduction}
We investigate partitioned factors, defined below, in circular Christoffel words and in Sturmian words. We classify and enumerate types of partitioned factors in Christoffel words, the varieties defined below, and determine the numbers of occurrences for each variety. Similarly, in Sturmian words we classify and enumerate the varieties of partitioned factors, and determine the frequencies (probabilities) for each variety.
 
\subsection{Definitions}

Consider symbols $a$, $b$, which we call \emph{letters} of an ordered alphabet $A=\{a<b\}$. Then the set of words over $A$ is the monoid, $$A^{\ast} = \{w=w_0 \dots w_{n-1} | w_i \in A, n \in \mathbb{N}\},$$  where the operation is concatenation of words, and the empty word $\varepsilon$ is assumed to belong to $A^{\ast}$ and serves as the monoid identity element.

\begin{defi}
Let $w =w_0 \dots w_{n-1} \in A^{\ast}$.
Then we say the \emph{length} of $w$ is $n$, and write $|w|=n$.
\end{defi}

\begin{defi}
If there exist $u,v \in A^{\ast}$ such that $w = uv$, then $u$ and $v$ are \emph{factors} of $w$. Furthermore, we say that $u$ is a \emph{prefix} of $w$ and $v$ is a \emph{suffix} of $w$.
\end{defi}

\begin{defi}
We say that $w$ and $w'$ are \emph{conjugates} of each other if $w = uv$ and $w' = vu$.
\end{defi}

\begin{defi}
Conjugacy is clearly an equivalence relation.
The set of conjugates of $w$ is called its conjugacy class, or equivalently, the \emph{circular word} $w$, notated  $(w)$.
\end{defi}

\begin{defi} Let $|w|_{b}$ be the number of occurrences of $b$ in $w$; we call this the \emph{height} of $w$. More generally, if $u$ is a factor of $w$, let $|w|_{u}$ be the number of occurrences of $u$ in $w$.
\end{defi}

Let $|w|=n>1$, and fix $m$, $1\leq m < n$.      

\begin{defi} 
We define $\mathcal{F}_m$, the \emph{multiset of circular factors of length $m$ of $w$}, to be the factors of length $m$ of the prefix of $w^2$ of length $n+m-1$.
\end{defi}

Let $1 \leq k \leq m$.

\begin{defi}
Let $P = (p_{1},\dots ,p_{k})$ be an ordered $k$-tuple of positive integers $p_{i}$, such that  $m=p_{1}+ \dots +p_{k}$.
We say that $P$ is a {\em{composition of $m$ into $k$ parts}}.
\end{defi}

\begin{defi}\label{parfac}
We define {\em{the multiset of $P$-partitioned factors of length $m$}} to be $\mathcal{F}_{(m,k)} = \{u \in \mathcal{F}_{m} | u=u_{1} \dots u_{k}\}$, where $|u_i |=p_{i}$. 
We call the factors $u_i$ the {\em{components}} of partitioned factor $u$.
\end{defi}

\begin{defi}
The \emph{height profile} of a partitioned factor $u$ is the sequence of integers $ |u_{i}|_{b}$, denoted by $\langle |u_{1}|_{b},\dots,  |u_{k}|_{b} \rangle$.
\end{defi}

\begin{defi}\label{vari}
We say that partitioned factors $u$, $u' \in \mathcal{F}_{(m,k)}$ are of the same \emph{variety} if their height profiles are equal, i.e., $|u_{i} |_{b}=|u'_{i}|_{b}$ for all $i$.
\end{defi}

\begin{defi}
The number of partitioned factors in $\mathcal{F}_{(m,k)}$ of a given variety is called the \emph{multiplicity} of that variety.
\end{defi}

We will also require the following two conventions. Given $x \in \mathbb{R}$, the ``floor function,'' $\lfloor x \rfloor$, indicates the greatest integer less than or equal to $x$: $x-1 <\lfloor x \rfloor \le x$. We use the notation $\{x\}$ to indicate the ``fractional part'' of $x$: $x - \lfloor x \rfloor = \{x\}$. Under this convention, $0 \le \{x\}<1$.

\begin{example}\label{six}
Consider the multiset of circular factors of length $m = 3$ in the word $w = aaabab$:
$$ \mathcal{F}_{3} = \{aaa, aab, aba, bab, aba, baa\}.$$
Take $k = 2$.
The composition $P = (1,2)$  determines the multiset of circular partitioned factors, 
$$ \mathcal{F}_{(3,2)} =  \{(a)(aa), (a)(ab), (a)(ba), (b)(ab), (a)(ba), (b)(aa)\}.$$

These partitioned factors exhibit four different height profiles, defining four different varieties: $\lambda_0 \leftrightarrow \langle 0,0 \rangle$, $\lambda_1 \leftrightarrow \langle 0,1\rangle$, $\lambda_2 \leftrightarrow \langle 1,0\rangle$
$\lambda_3 \leftrightarrow \langle 1,1\rangle$:
$$
\begin{array}{cc}
\text{Partitioned factor} &\text{variety}\\
\hline
(a)(aa) & \lambda_{0}\\
(a)(ab) & \lambda_{1}\\
(a)(ba) & \lambda_{1}\\
(b)(ab) & \lambda_{3}\\
(a)(ba) & \lambda_{1}\\
(b)(aa) & \lambda_{2}\\
\end{array}
$$
The multiplicities of ${\lambda_{0}}$, ${\lambda_{1}}$,  ${\lambda_{2}}$, and ${\lambda_{3}}$ are  1, 3, 1, 1, respectively. 
\end{example}

\begin{example}\label{seven}
Now consider 7-letter word $c = aabaaab$, and look at the multiset of circular factors of length $m = 4$:
$$\mathcal{F}_{4} = \{aaba, abaa, baaa, aaab, aaba, abaa,  baab\}.$$ 
With  $k = 2$ and $P = (1,3)$, we get the multiset of partitioned factors,
$$\mathcal{F}_{(4,2)} = \{(a)(aba), (a)(baa), (b)(aaa), (a)(aab), (a)(aba), (a)(baa),  (b)(aab)\}.$$ 

The partitioned factors are of three different height profiles giving rise to three varieties:
$\lambda_{0} \leftrightarrow \langle 0,1 \rangle$, $\lambda_{1} \leftrightarrow \langle 1,0  \rangle$,  $\lambda_{2} \leftrightarrow \langle 1,1 \rangle$:

$$
\begin{array}{cc}
\text{Partitioned factor} &\text{variety}\\
\hline
(a)(aba) & \lambda_{0}\\
(a)(baa) & \lambda_{0}\\
(b)(aaa) & \lambda_{1}\\
(a)(aab) & \lambda_{0}\\
(a)(aba) & \lambda_{0}\\
(a)(baa) & \lambda_{0}\\
(b)(aab) & \lambda_{2}\\
\end{array}
$$
The multiplicities of ${\lambda_{0}}$, ${\lambda_{1}}$, and ${\lambda_{2}}$ are  5, 1, and 1, respectively.
\end{example}

Note that, in Example \ref{six}, a composition into two components ($k = 2$) yielded four varieties, and in Example \ref{seven}, only three.  It can be verified by inspection that  \emph{any} set of partitioned factors  in $c = aabaaab$ will admit $k+1$ varieties for all $m$ and $k$ as defined. Theorem \ref{card} will outline conditions under which this property holds.

The word $c$ in Example \ref{seven} is a conjugate of a \emph{Christoffel word}, which may be defined as follows:
Given rational $\theta = q/(p+q)$ where  $(p,q)=1$, a word $w=w_0 \dots w_{p+q-1}$ is a (lower) Christoffel word if, for $i=0, \dots, p+q-1$: 
$$
w_i =
\begin{cases}
a \text{, if } \lfloor{(i+1)\theta}\rfloor - \lfloor{i\theta}\rfloor = 0,\\
b \text{, if }  \lfloor{(i+1)\theta}\rfloor - \lfloor{i\theta}\rfloor = 1.
\end{cases}
$$
We follow \cite{Berstel} and call the value $q/p$ the \emph{slope} of the Christoffel word. (\cite{Lothaire} considers the slope to be the height of a word divided by its length, here $\theta = q/(p+q)$, a definition that is more convenient for Sturmian words.) With $p = 5$ and $q = 2$, the value $\theta = 2/7$ generates the lower Christoffel word $aaabaab$ with slope  2/5, a conjugate of $c$ from Example \ref{seven}.

\subsection{Partitioned factors and their varieties}
We seek to demonstrate that if, and only if, for all $m,1 \leq m<n$, $\mathcal{F}_{(m,1)}$ admits two varieties, $\mathcal{F}_{(m,k)}$ admits $k+1$ varieties, for $1 \leq k \leq m$.  The ``only if'' direction is trivial, since the conclusion contains the premise. We will show that the premise implies that $w$ is a conjugate of a Christoffel word. Our theorem will assert an equivalence between being a Christoffel conjugate and having the property that $\mathcal{F}_{(m,k)}$ admits $k+1$ varieties. The multiplicities of the $k+1$ varieties will be determined in Corollary~\ref{mult}.

\begin{remark}
The \cite{BorelReutenauer} result is equivalent to the case $\mathcal{F}_{(m,m)}$ for $1 \leq m < n$: the composition is $1+1\dots +1=m$; there are $m+1$ classes. (Their theorem also includes the empty word: for length zero, there is one factor. We omit the trivial case of the empty factor.)
\end{remark}
  
\begin{prop}\label{mp}
 Given $(w)$, if for all $m,1 \leq m<n$, $\mathcal{F}_{(m,1)}$ admits two varieties, then $w$ is conjugate to a Christoffel word.
\end{prop}

\begin{proof}
We show that conjugates of $w$ are primitive and balanced$_{1}$, and thus are conjugates of a Christoffel word by \citet[Theorem 6.9]{Berstel}.

Conjugates of $w$ must be primitive, because if some conjugate $w'=t^{h}$ for $n>h>1$, factors of length $|t|<n$ would all have the same height, and $\mathcal{F}_{(t,1)}$ would admit a single variety.

Factors of length 1 are just $a$ and $b$: both letters must be factors. Assume there exist factors $u$, $v$ of length $m$, $1<m<n$, such that $||u|_{b}-|v|_{b}| \geq 2$.  Since not all factors of length $m$ have the same height, there must be a conjugate of $w$ with a factor $u' = u_{1}u_{2}\dots u_{m}$, where $|u'|_{b} = |u|_{b}$, and such that there exists $v' =u_{2} \dots u_{m}u_{m+1}$ of different height from $u'$.  But $v'$ differs from $u'$ in exactly one letter, so $||u'|_{b}-|v'|_{b}|=1$. Together with $||u|_{b}-|v|_{b}| \geq 2$, $\mathcal{F}_{(m,1)}$ would admit at least 3 varieties, contrary to hypothesis.   

Therefore, the circular word $(w)$ is balanced$_{1}$ and primitive (i.e., $w$ and all of its conjugates are balanced$_{1}$ and primitive).  By \cite{Berstel}, $w$ is conjugate of a Christoffel word. 
\end{proof}

Let us assume, then, that $w$ is a Christoffel word of length $n>1$, and we will consider it together with all its conjugates (equivalently, the circular word $(w)$). Further, we may take $w$ to be a lower Christoffel word of slope ${q}/{p}$. Then $p+q=n$, $(p,q)=1$, equivalently, $(p,n) = (q,n) = 1$. Let $p^{\ast} = p^{-1} \bmod n$; $q^{\ast} = q^{-1} \bmod n$. We take $p>q$, where $|w|_{a} = p$, $|w|_{b} = q$. \\

\begin{theorem}\label{card}
If and only if $w$ is conjugate to a Christoffel word, $\mathcal{F}_{(m,k)}$ admits $k+1$ varieties, for $1 \leq k \leq m$.  
\end{theorem}

\begin{proof}
$\Leftarrow$ The converse is trivial by Prop. \ref{mp}. 

$\Rightarrow$ We make use of the Burrows-Wheeler Transform and the matrix BWT$_{w}$. 
By construction, the rows of BWT$_{w}$ are the conjugates $w_{(i)}$ of $w$ in lexicographic order: 
$w_{(0)}< w_{(1)} < \dots < w_{(n-1)}$. See \cite{MRS}.
(We write $w_{(i)}$ to denote the $i$-th conjugate of $w$ in lexicographic order, and $w_{i}$ to denote the $i$-th letter of $w_{(0)}$.)
Let $C^{i}$ be the conjugation operator that rotates a given word by $i$ letters. For example, $C^{i}(w_{(0)}) = w_{i} w_{i+1}\dots w_{i-1}$.

We appeal to Lemma~\ref{isc}, due to \cite{MRS} and  \cite{BorelReutenauer}.
The lemma forms the heart of our proof.
\begin{lemma}\label{isc}
For each $i = 1, \dots, n-1$  and for words $u,v$, one has $w_{(i-1)} = uabv, w_{(i)} = ubav$. Also $w_{(i)} = C^{q^{\ast}}(w_{(i-1)}) = C^{iq^{\ast}}(w_{(0)})$, and $|ua| = |ub| = ip^{\ast} \bmod n$.
\end{lemma}

 It follows from the lemma that the $j$th letter in the $i$th row is $w_{iq^{\ast}+j}$. (Going forward, all calculations are reduced to least non-negative values modulo $n$.) 

We wish to determine the indices of the letters that participate in the reversal $ab$ to $ba$ uniquely distinguishing rows $i-1$ and $i$. In the following lemma, bear in mind that the letters of conjugates are indexed with respect to the lexicographically least element, $w_{(0)}$.

\begin{lemma} For all $i$ from 1 to $n-1$: in $w_{(i-1)}$ the factor $ab$, that uniquely is exchanged with the factor $ba$ in $w_{(i)}$, is indexed by $p^{\ast}-1$ and $p^{\ast}$. The corresponding factor $ba$ in $w_{(i)}$ is indexed by $n-1$ and $0$.
\end{lemma}

\begin{proof} Lemma 1 implies that the column position for the first letter in the $ab \leftrightarrow ba$ reversal in rows $i-1$ and $i$ is $j=|ua|-1 = ip^{\ast}-1$. Therefore we have the following arrangement:
$$
\begin{array}{rrr}
w_{(i-1)}: & \dots  w_{(i-1)q^{\ast}+(ip^{\ast}-1)} = a & w_{(i-1)q^{\ast}+ip^{\ast}} = b  \dots\\
w_{(i)}: & \dots  w_{(i)q^{\ast}+(ip^{\ast}-1)} = b & w_{(i)q^{\ast}+ip^{\ast}} = a  \dots\\
\end{array}
$$
The indices simplify as follows:
\begin{enumerate}[label=(\roman*)]
\item$(i-1)q^{\ast}+(ip^{\ast}-1) \equiv p^{\ast}-1.$\label{1}
\item$(i-1)q^{\ast}+ip^{\ast} \equiv p^{\ast}.$\label{2}
\item$iq^{\ast}+(ip^{\ast}-1) \equiv n-1$.\label{3}
\item$iq^{\ast}+ip^{\ast} \equiv 0.$\label{4}
\end{enumerate}

We begin by demonstrating \ref{2}.\\

\begin{tabular}{lrcll}
\ref{2}: & $(i-1)q^{\ast}+ip^{\ast}$ & $\equiv$ & $i(q^{\ast}+p^{\ast})-q^{\ast}$ & \\ 
 & &$\equiv$ & $-q^{\ast}$ & $p + q = n$ \text{, thus }$q^{\ast}+p^{\ast} = n$ \\
& &$\equiv$ &$(-q)^{\ast}$ & \\
& &$\equiv$ &$p^{\ast}$ & because $p^{\ast} \equiv -q^{\ast} \bmod n$.\\
& Therefore &&&\\
\ref{1}: & $(i-1)q^{\ast}+(ip^{\ast}-1)$ & $\equiv$ & $p^{\ast}-1$ & \\ 
\end{tabular}\\

The corresponding indices of letters in row $i$ that form the factor $ba$ are, respectively, $n-1$ and $0$:

\begin{tabular}{lrcl}
\ref{3}: & $iq^{\ast}+(ip^{\ast}-1)$ & $\equiv$ & $ (iq^{\ast}+ip^{\ast})-1$  \\ 
&& $\equiv$ & $n-1$.\\
 And so, && &\\
\ref{4}: & $iq^{\ast}+ip^{\ast}$ &$\equiv$ & 0\\
\end{tabular}\\

We therefore have $w_0 = w_{p^{\ast}-1} = a$, and $w_{p^{\ast}} = w_{n-1} = b$. We find these reversed pairs of letters in the same two columns of adjacent rows: 

$$
\begin{matrix}			
w_{(i-1)}: & \dots w_{p^{\ast}-1} = a & w_{p^{\ast}} = b \dots\\
w_{(i)}: & \dots w_{n-1} = b & w_{0} = a \dots\\
\end{matrix}
$$ 
\end{proof}

These four letters also occupy the corners of the matrix.
Clearly, $w_{0}$ and $w_{n-1}$ are the first and last letters of $w_{(0)}$.
To determine the first and last letters of $w_{n-1}$ we apply our knowledge from Lemma \ref{isc} of $w_{iq^{\ast}+j}$: for $i=n-1$, and $j=0$, $w_{(n-1)q^{\ast}+0} = w_{p^{\ast}} = b$, and for $j=n-1$, $w_{(n-1)q^{\ast}+n-1} = w_{p^{\ast}-1} = a$. \citet[47--49]{Berstel} prove that $w_{(0)}$ and $w_{(n-1)}$ form a corresponding pair of lower and  upper Christoffel words, $aub$ and $bua$ respectively. Therefore, as in the case of adjacent rows, the first and last rows also fix $n-2$ letters and reverse the other two.
$$
\begin{matrix}
w_{(0)}: &  w_{0} = a &u &  w_{n-1} = b\\
\vdots & \vdots  & \cdots & \vdots \\
w_{(n-1)}: & w_{p^{\ast}} = b &u & w_{p^{\ast}-1} = a\\
\end{matrix}
$$

From Lemma~\ref{isc} two adjacent conjugates $w_{(i)}$ and $w_{(i+1)}$ in lexicographic order are of the form $uabv$ and $ubav$, respectively. That is, $n-2$ letters remain fixed and two letters exchange positions. We employ the letter reversals to locate the lexicographically least element of each variety. Consider the prefixes of length $m$ of the conjugates that form BWT$_{w}$, partitioned according to $P$;  these partitioned prefixes comprise the multiset $\mathcal{F}_{(m,k)}$. 

Let $u_{(i)}$ and $u_{(i+1)}$ be members of $\mathcal{F}_{(m,k)}$ associated with the conjugates $w_{(i)}$ and $w_{(i+1)}$ in adjacent rows of the matrix.
Let suf$_{(i)}$ be the suffix of $w_{(i)}$ such that  $ w_{(i)} = u_{(i)}\text{suf}_{(i)}$ (analogously, suf$_{(i+1)}$). Because $m < n$, suf$_{(i)} \ne \epsilon$. If $|\text{suf}_{(i)}|>1$, then the reversal of letters may appear entirely within suf$_{(i)}$ and suf$_{(i+1)}$. When this happens, then $u_{(i)} = u_{(i+1)}$, and so $u_{(i)}$ and $u_{(i+1)}$ belong to the same variety.

Consider the components of $u_{(i)}$. If the factor $ab=w_{p^{\ast}-1}w_{p^{\ast}}$ occurs entirely within one of the components, that component must be at least of length 2, and the reversal occurs within the corresponding component of $u_{(i+1)}$. Then the components have the same height. Because the heights of all the other components remain the same, the partitioned factors $u_{(i)}$ and $u_{(i+1)}$ are of the same variety. 

	On the other hand, if the reversal of letters crosses a boundary between components, then $u_{(i)}$ and $u_{(i+1)}$ belong to different varieties. Given the $c$-th component of $u_{(i)}$, $1 \leq c < k$, with last letter $w_{p^{\ast}-1}=a$, and the $c+1$-th component of $u_{(i)}$ with first letter $w_{p^{\ast}}=b$, then in $u_{(i+1)}$ the last letter of the $c$-th component is $w_{n-1}= b$, and the first letter of the $c+1$-th component is $w_{0} = a$, and the heights of the corresponding components differ. By definition, $u_{(i)}$ and $u_{(i+1)}$ are of different varieties.

	Finally, the reversal may take place between the last letters of the $k$-th components and the first letters of the respective suffixes: the last letter of the $k$-th component of $u_{(i)}$ is $a$, the first letter of suf$_{(i)} = b$, the last letter of the $k$-th component of $u_{(i+1)}$ is $b$, the first letter of suf$_{(i+1)} = a$. In this case, the heights of the $k$-th components differ by 1, and so $u_{(i)}$ and $u_{(i+1)}$ belong to different varieties. 

In all cases where $u_{(i)}$ and $u_{(i+1)}$ are of different varieties, $u_{(i+1)}$ is lexicographically least of its variety. (See Figure~\ref{generic}.)

 We form partial sums $s_\ell$ from the composition of $m$, always beginning with 0: $s_0 = 0, s_1=p_1, s_2 = p_1+p_2, \dots,  s_k=p_1+ \dots + p_k$. The columns of the matrix indexed by the partial sums mark the boundaries between components in the partitioned factors. By the above argument, the lexicographically least element of any variety appears exactly when letter $w_0$ appears in the $s_\ell$-th column. Because there are $k+1$ values for $s_\ell$, there are exactly $k+1$ varieties.
\end{proof}

\begin{figure}
{
\centering
\begin{tabular}
{c | ccc | ccc  |       |l}
 1 & &  2 & & &  3 &\\
\hline
\dots & \dots & \dots & \dots & \dots & \dots& \dots & \\
\dots & \dots & $w_{p^{\ast}-1} = a$ & $w_{p^{\ast}}= b$ &\dots & \dots & \dots  & variety $\alpha$ \\
\dots & \dots & $w_{n-1}=b$ & $w_{0}=a$ & \dots& \dots & \dots  & variety $\alpha$ \\
\dots & \dots & \dots & & \dots & \dots  & \dots  & \\
\dots & \dots & \dots & $w_{p^{\ast}-1}=a$ & $b_{p^{\ast}}=b$ & \dots  & \dots& variety $\alpha$\\
\hline
\dots & \dots & \dots & $w_{n-1}=b$ & $w_{0}=a$ & \dots  &  \dots& variety $\beta$\\
\dots & \dots & \dots & & \dots & \dots  & \dots &  \\
\dots & \dots & \dots & \dots & \dots & \dots & $w_{p^{\ast}-1} = a$  & variety $\beta$\\
\hline
\dots & \dots & \dots & \dots & \dots & \dots & $w_{n-1}=b$  & variety $\gamma$\\
\end{tabular}
\caption{Letter reversals and varieties in $\mathcal{F}_{(m,3)}$. When reversal $ab\leftrightarrow ba$ occurs within components, the adjacent partitioned factors belong to the same variety. When it crosses a partition boundary, the adjacent partitioned factors belong to different varieties. Let the height profile of variety $\alpha$ be $\langle r, s, t \rangle$; then for variety $\beta$ it is $\langle r, s+1, t-1 \rangle$ and for variety $\gamma$, $\langle r, s+1, t \rangle$.
\label{generic}
}}
\end{figure}

The theorem is illustrated in Figure~\ref{fig: Partitioned_Christoffel}, where $w = aaabaab$ is the lower Christoffel word of slope ${2}/{5}$ and the set $\mathcal{F}_{(4, 3)}$ of partitioned factors of length four is determined by a composition of four into three components, i.e., $P = (1,2,1)$. This in turn gives rise to a composition $\Pi$ of $7$ that yields the multiplicities, $(1,4,1,1)$. Determining multiplicities forms the focus of Corollary~\ref{mult}.

\subsection{Multiplicities of varieties}
Theorem \ref{card} implies that the varieties appear partitioned horizontally in lexicographic order in the BWT matrix. Following this order, we define the varieties as $\lambda_{\ell}$ with $0 \le \ell \le k$. Variety $\lambda_{0}$ always appears in row 0. In general, the lexicographically least element of a variety appears in a row in which letter $w_{0}$ is in some column $s_{\ell}$. We calculate these row numbers by solving for $i$ in the congruence $iq^{\ast}+j \equiv 0$ when $j = s_{\ell}$. By Lemma~\ref{isc}, $ip^{\ast}  \equiv s_{\ell}$ and so,
\begin{equation}\label{ival}
i \equiv s_{\ell}p. 
\end{equation}

There are $k+1$ values of $s_{\ell}$ and, because $(n,p) = 1$, Congruence \ref{ival} has $k+1$ solutions. We order these solutions:
Again, let $i_{\ell} = s_{\ell}p$. Let $\sigma$ be a permutation of the $k+1$ values $i_{\ell}$ such that $0 = \sigma(i_{0}) < \sigma(i_{1}) < \dots < \sigma(i_{k})$. Then, for $0 \leq \ell \leq k$, let $\pi_{\ell} = \sigma(i_{\ell+1})-\sigma(i_{\ell})$ defining $\pi_{k}$ as $n - {\sigma(i_k)}$.
Finally, let $\Pi = (\pi_{0},\dots,\pi_{k})$. By construction, $\Pi$ is a composition of $n$ into $k+1$ components.

\begin{cor}\label{mult}
Given $w$ conjugate to a Christoffel word, the multiplicities of the $k+1$ varieties $(\lambda_{0},\dots, \lambda_{k})$ of $\mathcal{F}_{(m,k)}$ are given by the elements of the composition $\Pi= (\pi_{0},\dots,\pi_{k})$.
\end{cor}

\begin{proof}
By  Congruence~\ref{ival}, the row numbers of the first appearances of the varieties are the solutions to $i = s_{\ell}p$, which are ordered, $0 =\sigma(i_0) < \sigma(i_1) < \dots < \sigma(i_k)$. The first variety, $\lambda_{0}$, appears in row $\sigma(i_0) = 0$ and has multiplicity $\sigma(i_1) - \sigma(i_0)$. In general, then,  as $\ell$ ranges from 0 to $k$, $\lambda_{\ell}$ first appears in row $ \sigma(i_\ell)$ and has multiplicity $\pi_{\ell} = \sigma(i_{\ell + 1}) - \sigma(i_\ell).$
The sequence of differences $\pi_{\ell}$ gives the composition $\Pi$.
\end{proof}

\begin{figure} 
\center  
\includegraphics [width=13 cm] {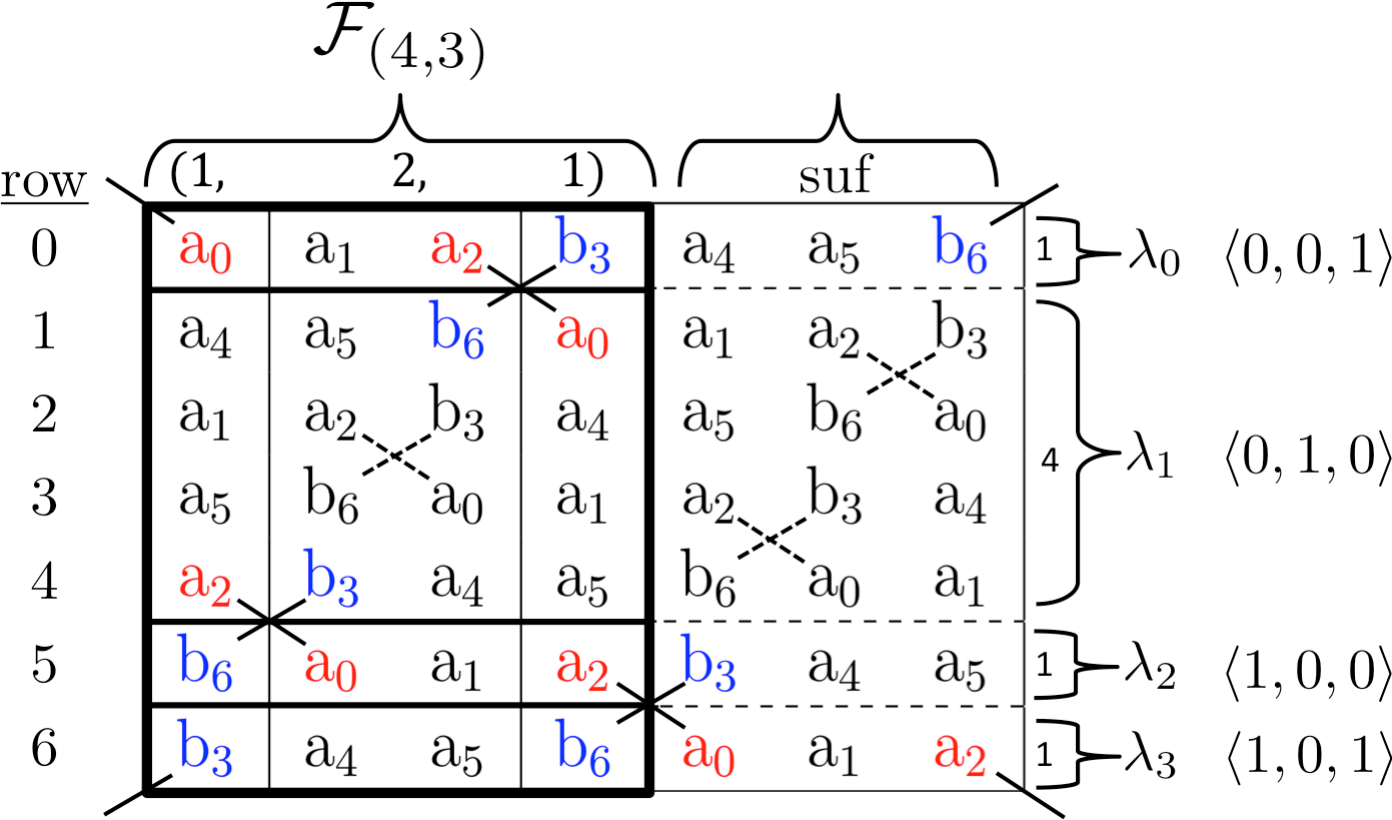}   
\caption{In $(w) = (aaabaab)$, $\mathcal{F}_{(4,3)} $  determined by composition (1,2,1). The four varieties have height profiles $\langle 0,0,1 \rangle$,  $\langle 0,1,0 \rangle$,  $\langle 1,0,0 \rangle$, $\langle 1,0,1 \rangle$. The multiplicities of these varieties are produced by composition $\Pi = (1,4,1,1)$.} 
\label{fig: Partitioned_Christoffel}
\end{figure}

Turning again to Figure~\ref{fig: Partitioned_Christoffel}, partial sums $s_{\ell} = \{0,1,3,4\}$ of the composition $P = (1,2,1)$ determine the four varieties, $\lambda_{0}, \lambda_{1}, \lambda_{2}, \lambda_{3}$. The lexicographically least element of each variety appears when $w_{0}$ is in the column immediately to the right of a partition boundary, namely those columns associated with partial sums $s_{\ell}$. The first variety is (always) produced by the appearance of $w_{0}=a$ in row zero, column zero. In this example, variety $\lambda_{0}$ is characterized by the height profile, $\langle 0,0,1 \rangle$. In row $1$, $w_{0}$ appears in column 3, immediately to the right of a partition boundary, giving rise to variety $\lambda_{1}$, with a height profile of $\langle 0,1,0 \rangle$. The next appearance of $w_{0}$ to the right of a partition boundary is column 1 in row 5, generating the variety $\lambda_{2}$, $\langle 1, 0, 0 \rangle$. Finally, it appears in column 4 of row 6, giving variety $\lambda_{3}$, $\langle 1,0,1 \rangle$. We may calculate the row numbers directly by taking multiples $ s_{\ell}p\pmod 7$: $0\times 5 = 0$; $1\times 5 = 5$; $3\times 5 = 1$; $4\times 5 = 6$. The permutation $\sigma$ orders these solutions, $(0,1,5,6)$: Taking circular differences modulo 7  forms the composition $\Pi = (1,4,1,1)$, which yields the multiplicities of the four varieties. 

\subsection{The case of $\mathcal{F}_{(m,m)}$}
  An application to the case of circular factors of length $m$ of a Christoffel word of slope $q/p$    specializes to the composition of $m$ into $m$  parts, $m=1+1+\dots+1$. There are $m+1$ distinct factors, as is known, and the result shows that therefore the multiplicities of the occurrences are determined by taking least nonnegative residues mod $n$ of $0  p$, $1  p$, $2  p, \dots, m  p$, ordering them as ordinary integers, adjoining $n$ as greatest element, and taking differences between adjacent elements. Figure~\ref{fig:1111} illustrates with $\mathcal{F}_{(4,4)}$ in $aabaabab$, the Christoffel word of slope 3/5. The five varieties determined by the composition $P = (1,1,1,1)$ have multiplicities $\Pi = (2,2,1,2,1)$.

For proofs of Theorem \ref{card} and Corollary \ref{mult} in a different mathematical environment and an application to music, see  \cite{CloughMyerson85, CloughMyerson86}.
 
\begin{figure} 
\center  
\includegraphics [width=13 cm] {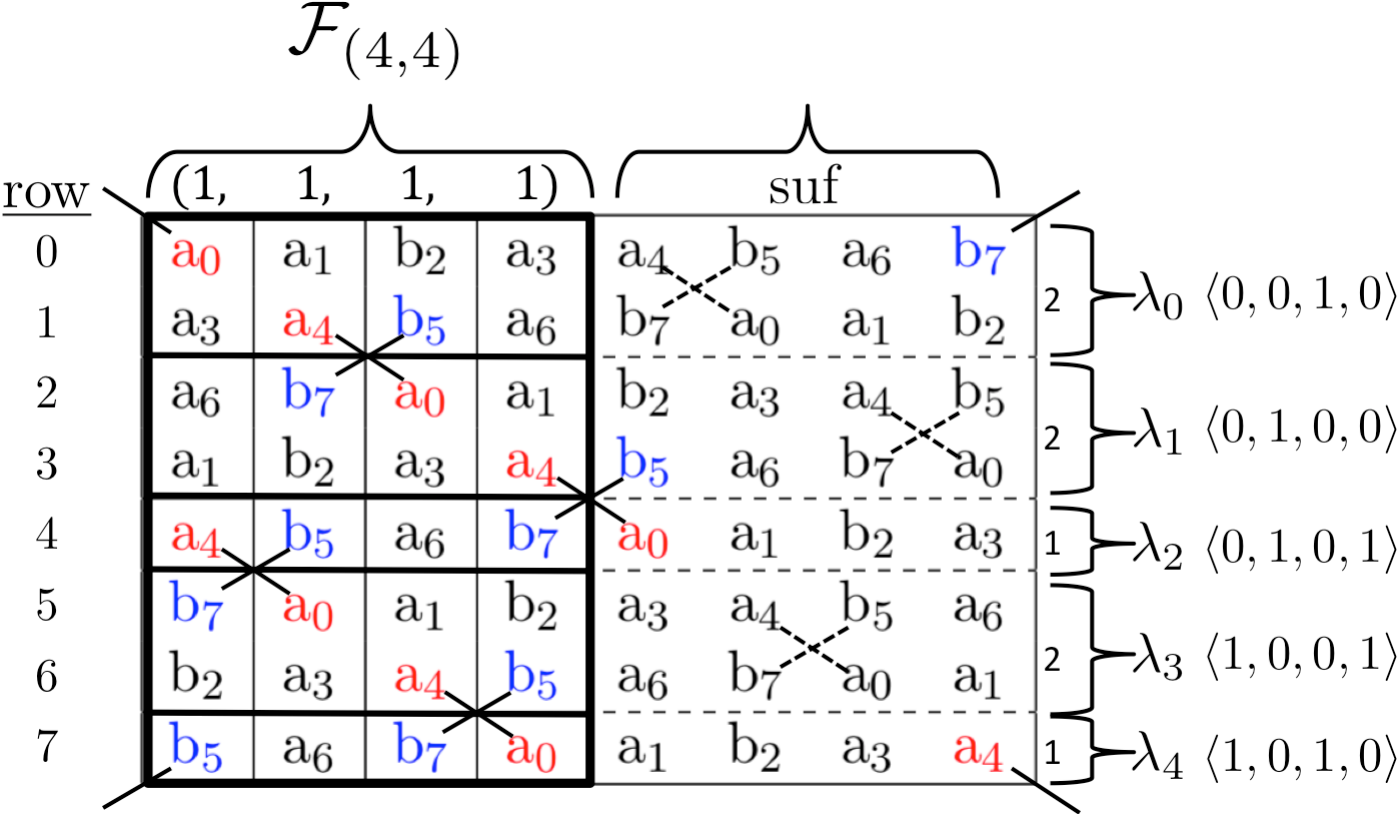}   
\caption{In  $(w) = (aabaabab)$, $\mathcal{F}_{(4,4)}$  determined by composition (1,1,1,1). The five varieties have height profiles $\langle 0,0,1,0 \rangle$,  $\langle 0,1,0,0 \rangle$,  $\langle 0,1,0,1\rangle$, $\langle 1,0,0,1 \rangle$, $\langle 1,0,1,0 \rangle$. The multiplicities of these varieties are produced by composition $\Pi = (2,2,1,2,1)$.} 
\label{fig:1111}
\end{figure}

\section{Partitioned Factors in Sturmian Words}\label{pfsturm}

We now turn to the question of the varieties and frequencies (probabilities) of partitioned factors of a Sturmian word $s$ over $\{a<b\}$.  A Sturmian word is an infinite word that has $n+1$ factors of length $n$ for every $n \geq 0$. An explicit arithmetic construction, analogous to the one given for Christoffel words, will be presented in Equation (\ref{sturdef}).

Let the infinite word $s=s_{0}s_{1} \dots$, where $s_{i} \in \{a<b\}, i \in \mathbb{N}$, be Sturmian.

\begin{defi}
A word $u \in \{a<b\}^{\ast}$ of length $m$ is a \emph{factor} of $s$ if, for some $j \in \mathbb{N}$, $u=s_{j}\dots s_{j+m-1}$.
\end{defi}

\begin{defi}
The set of factors of $s$ of length $m$ is defined as $\mathcal{L}_{m}$. A Sturmian word is characterized by the fact that the cardinality of $\mathcal{L}_{m}$ is $m+1$ for all $m \in \mathbb{N}$. 
\end{defi}

\begin{defi}
Following Definition \ref{parfac}, the set of  $P$-\emph{partitioned factors of length} $m$, $\mathcal{L}_{(m,k)}$, consists of the factorizations $u=u_{1} \dots u_{k}$, where $|u_{i}| = p_{i}$, $1 \leq i \leq k$. 
\end{defi}

\begin{defi}
Following Definition \ref{vari}, partitioned factors $u$, $u'$ in $\mathcal{L}_{(m,k)}$ are of the same variety if their height profiles are equal. Being of the same variety is clearly an equivalence relation on $\mathcal{L}_{(m,k)}$; we identify the varieties with the equivalence classes.  
\end{defi}

\begin{defi}
Let $v$ be a prefix of $s$ such that $|v| = n$. Following \cite{Berthe}, the \emph{frequency} of $w$ in $s$ is 
$$Fr(w)=\lim_{{n} \to \infty} \frac{|v|_w}{n}.$$
\end{defi}

With $n = |v| \geq m$, let $\mu_{n}$ be the number of occurrences of $P$-partitioned factors of length $m$ in $v$ that are of variety $\lambda$.

\begin{defi}
We define  the \emph{frequency of $\lambda$ in $s$},
$$Fr(\lambda) = \lim_{n\to\infty} \frac{\mu_{n}}{n}.$$
That is, $Fr(\lambda)$, is this limit, if it exists.
\end{defi} 

\begin{example}\label{l23}
Let $s_{log_{2}(3/2)}$ be the right infinite mechanical word of slope $\log_{2}(3/2)$,
$$s_{log_{2}(3/2)} = {abababbababbabababbababbabababbababb\dots,}$$
and consider $\mathcal{L}_{4} = \{abab, abba, baba, babb, bbab\}$ under the composition (1,3). This is $\mathcal{L}_{(4,2)} = \{(a)(bab), (a)(bba), (b)(aba), (b)(abb), (b)(bab)\}$. The elements in this set exhibit three height profiles, $\langle 0,2  \rangle$, $\langle 1,1 \rangle$, and $\langle 1,2 \rangle$. These determine three equivalence classes, $\lambda_{0}$, $\lambda_{1}$,  and $\lambda_{2}$, which partition $\mathcal{L}_{(4,2)}$:
$$\lambda_{0}=\{(a)(bab), (a)(bba)\}, \lambda_{1}= \{(b)(aba)\}, \lambda_{2}= \{(b)(abb), (b)(bab)\}.$$
\end{example}

\begin{remark}\label{evfac} Every factor of $s$ is a proper factor of some prefix of $s$ which is a Christoffel word. Therefore, Theorem \ref{card} applies, and we know already that the number of varieties in $\mathcal{L}_{(m,k)}$ is $k+1$. But we rehearse the argument in Sturmian words in order to determine the frequencies of the varieties.
\end{remark}  

Since the underlying factors are fundamental, we begin with a demonstration of the minimal complexity of $s_{\theta}$, that there are $m+1$ distinct factors of length $m$, for all $m \in \mathbb{N}$.

\subsection{Varieties and Frequencies of Factors}\label{vff}
   
We may identify Sturmian words with aperiodic (lower) mechanical words of irrational slope $\theta$ and non-negative offset or intercept $\rho$, following \cite{Lothaire}: the sequence over ordered alphabet $\{a<b\}$ is 

\begin{equation}\label{sturdef}
s_{\theta,\rho}(n) =
\begin{cases}
a, & \text{if }\lfloor(n+1)\theta+\rho \rfloor - \lfloor n\theta+\rho \rfloor= 0\\
b, &  \text{if }\lfloor(n+1)\theta+\rho \rfloor - \lfloor n\theta+\rho \rfloor= 1\\
\end{cases}
\end{equation}
for irrational $\theta$, $0<\theta<1$, $n \in \mathbb{Z}$. Since the offset value $\rho$ has no effect on the set of factors \cite[53]{Lothaire}, we may take $\rho = 0$. This means that, for a slope $\theta$ and offset 0, the right-infinite mechanical word is $ac_{\theta}$ and the reversal of the left-infinite mechanical word is $bc_{\theta}$, where $c_{\theta}$ is the \emph{characteristic word} of slope $\theta$ \cite[55]{Lothaire}. That is, the bi-infinite Sturmian word of slope $\theta$ and offset 0 is symmetrical about the word $ba$. 

Because $\lfloor(n+1)\theta \rfloor -  \lfloor n\theta \rfloor = 1 \iff \{n\theta\} \geq 1-\theta$, we may equivalently encode a Sturmian word $s_{\theta}$ according to which segment of the unit circle the fractional part of a multiple of $\theta$ lies in. We write the $n$th letter as

\[s_{\theta}(n) = s_{n} =
\begin{cases}
a, & \text{if }\{n\theta \} \in I_{a} = [0, 1-\theta)\\
b, &  \text{if }\{n\theta \} \in I_{b} =  [1-\theta, 1).\\
\end{cases}
\]
   
The following rehearses the treatment of encoding factors by rotations found in \citet[50]{Lothaire}. Let $x$ be a point on the half-open unit interval $[0, 1)$, identified with the unit circle.  

\begin{defi}\label{rota} Rotation by angle $\theta$ maps $[0,1)$ to itself by $R_{\theta}(x) = R(x) =: \{x+\theta \}$. 
\end{defi}

Composing $R$ with itself $n$ times we have $R^{n}(x) = \{x+n\theta \}$. We may, letting $x$ range over some subinterval of $[0, 1)$, map a subinterval to a subinterval, with the understanding that a subinterval may straddle 0, i.e., we understand the half-open subinterval $[s, t)$ where $0 \leq t < s < 1$ to be the union $[s, 1) \cup [0, t)$. 

Consider a word of length $m$ over $\{a,b\}$, $w = r_{0}r_{1} \ldots r_{m-1}$. When is this word a factor of $s_{\theta} = s_{0}s{_1} \ldots$? (It suffices to look at the right-infinite side, by symmetry.) Then $w$ is a factor if and only if for some $n\geq 0$, $r_{0}=s_{n}$, $r_{1}=s_{n+1}, \ldots, r_{m-1}=s_{n+m-1}$. 
 By the encoding defined above, $w$ is a factor if and only if for all $i$ from 0 to $m-1$, $R^{n+i}(0) \in I_{a}$ when $r_{i}=a, R^{n+i}(0) \in I_{b}$ when $r_{i}=b$. We define $I_{r_{i}} =: I_{a}$ or $I_{b}$ according to whether $r_{i} = a$ or $b$, respectively. Then we may write, equivalently, taking $R^{-i}$ on both sides, $w$ is a factor if and only if $R^{n}(0) \in R^{-i}(I_{r_{i}})$, for all $i$ from 0 to $m-1$. That is, $w=s_{n}s_{n+1} \ldots s_{n+m-1} \iff R^{n}(0) \in I_{r_{0}}\cap R^{-1}(I_{r_{1}})\cap \ldots \cap R^{-m+1} (I_{r_{m-1}})$.
 
\begin{theorem}\label{fact} \emph{(Lothaire)} The factors ${w}$ of length ${m}$ of $s_{\theta}$ are identified with the ${m}+1$ subintervals defined by the points $0, \{-1\theta \}, \{-2\theta \}, \ldots, \{-m\theta \}$, in some order given by a permutation of the ${m}$ indices.
\end{theorem}

\begin{proof} By the characterization above,  $w = r_{0}r_{1} \ldots r_{m-1}$ is a factor of $s_{\theta}$ where $r_{0}=s_{n}$ exactly when $\bigcap_{i=0}^{m-1}R^{-i}(I_{r_{i}})$ is non-empty, and when the initial point $R^{n}(0)$ lies in the subinterval that is this intersection. How many intersections are there for factors of length $m$, and what points define them? The endpoints defined by the 0th rotation of one of the subintervals $I_{a}$, $I_{b}$ are 0 and $1-\theta$, one or the other point excluded. Each of the negative rotations of one of the subintervals define two endpoints, $\{-i\theta\}, \{-(i+1)\theta\}$, since the point $1-\theta$ is already $\{-1\theta\}$ (so $R^{-i}(1-\theta) = \{-(i+1) \theta \}$). One of those endpoints is included, the other excluded; that is to say, $R^{-i}(I_{a})$ and $R^{-i}(I_{b})$ partition the unit circle into two disjoint intervals. The intersections are then the smallest subintervals defined by the $m+1$ points 0, $1-\theta = \{-1\theta \}, \{-2\theta \}, \ldots, \{-m\theta \}$, in some order given by a permutation of the indices 1 to $m$. That is, the intersections are a partition of the unit circle into $m+1$ disjoint subintervals demarcated by these $m+1$ points. Each of the appearances in $s_{\theta}$ of a given factor $w$ of length $m$ must correspond one-to-one to a point $\{n\theta\}$ in exactly one of these subintervals, and each of the countably infinite points $\{n\theta\}$ corresponds to the initial letter of one of the factors of length $m$, depending on which subinterval it falls into.
\end{proof}

\begin{remark}\label{charstur} Note that one of the characterization properties of a Sturmian word, that factors of length \emph{m} come in \emph{m}$+1$ varieties, falls out of this result. 
\end{remark}

\begin{cor}\label{freqfact} The frequencies of each factor are the respective lengths of the subintervals, which take on at most three distinct values.
\end{cor}
\begin{proof}
This follows from the uniform distribution property. Let $v$ be a prefix of $s$ with $|v|=N \geq m$. By Theorem \ref{fact} we have the number of occurrences $|v|_{w}$ is exactly the number of integers $n$, $0 \leq n < N$, such that $R^{n}(0) \in \bigcap_{i=0}^{m-1}R^{-i}(I_{r_{i}})$. From \cite{weyl}, the values $R^{n}(0) = \{n\theta\}$ for all $n \in \mathbb{N}$ are uniformly distributed on the interval $[0, 1)$. Then $Fr(w)=\lim_{{N} \to \infty} \frac{|v|_w}{N}$ exists, and is the length of the subinterval of intersection. By the three-distance theorem, as \cite{Berthe} shows, these frequencies can take on at most three distinct values.
\end{proof}
 
The extension to frequencies of partitioned factors is immediate: the varieties of partitioned factors are the equivalence classes defined above, which are given by equivalence classes that partition the set of factors of length $m$. The frequencies of the factors are given by the lengths of the intersections, so the frequency of a partitioned factor variety is given by taking the underlying factors in that equivalence class and summing the lengths of their respective intersections.

\begin{cor}\label{parfacs} Given variety $\lambda \in \mathcal{L}_{(m,k)}$, let $L \subset \mathcal{L}_{m}$ such that $u \in L \iff$ the $P$-partitioned factor $u$ is in the equivalence class of $\lambda$. Then $Fr(\lambda) = \sum_{u \in L} Fr(u)$.      
\end{cor}
\begin{proof}
Variety $\lambda$ is an equivalence class of $P$-partitioned factors in the set $\mathcal{L}_{(m,k)}$ with identical height profiles. This class comprises the set of underlying factors, $L$, which by Theorem \ref{fact} are associated with subintervals of $[0, 1)$ determined by the points $0, \{-\theta\}, \{-2\theta\}, \dots, \{-m\theta\}$. By Corollary \ref{freqfact}, the frequencies $Fr(u)$ of the underlying factors $u$ are the lengths of the associated subintervals. It follows that the frequency of variety $\lambda$ is the sum of the lengths of the disjoint subintervals, i.e., $Fr(\lambda) = \sum_{u \in L} Fr(u)$.  
\end{proof}

How does the lexicographic order of the factors in $\mathcal{L}_{m}$ and therefore of partitioned factors in $\mathcal{L}_{(m,k)}$ relate to the order of the intervals as encoded by rotations?    

\subsection{Lexicographic Ordering of Factors} 

As noted in Remark~\ref{evfac} above, any factor of $s_{\theta}$ is a factor of a prefix of $s_{\theta}$ that is a Christoffel word (in fact, a factor of an infinite number of Christoffel prefixes). As such, the number of varieties of factors of a given length $m$ is $m+1$, and the number of varieties of partitioned factors in $\mathcal{L}_{(m,k)}$ is $k+1$. We also know from Theorem~\ref{card} that when the factors are ordered lexicographically, the varieties of partitioned factors inherit the lexicographic order of the underlying factors, gathering together lexicographically adjacent members of $\mathcal{L}_{m}$ in the equivalence classes which are the varieties. To show that indeed all of these determinations are independent of the given Christoffel prefix,  and to answer the question of the relationship between the lexicographic order of the factors and the order of the subintervals of $[0, 1)$ corresponding to the factors, we rehearse the arguments in terms of the encoding via rotations.

It is evident that the mappings $R^{-i}$ send the complementary intervals $I_{a} = [0, 1-\theta)$ and $I_{b} = [1-\theta, 1)$ to images that are complementary intervals. A notation that will be useful makes this explicit: for all $i, 0 < i \leq n$,

$$
\begin{array}{ll}
R^{-i+1}(I_{a}) &= R^{-i+1}([0, 1-\theta))\\
& = R^{-i+1}([0, \{-\theta \}))\\
& = R^{-i+1}([0, R^{-1}(0)) \\
& = [R^{-i+1}(0), R^{-i}(0)),\\
\end{array}
$$
 and similarly,

 $$
\begin{array}{ll}
R^{-i+1}(I_{b}) &= R^{-i+1}([1-\theta, 1))\\
&= R^{-i+1}([\{-\theta \}, 1))\\
&= R^{-i+1}([R^{-1}(0), 1)) \\
&= [R^{-i}(0), R^{-i+1}(0)).\\
\end{array}
$$

We could also rewrite these complementary intervals then as $$R^{-i+1}(I_{a}) = [\{(-i+1)\theta \}, \{-i\theta \})$$ and $$R^{-i+1}(I_{b}) = [\{-i\theta \}, \{(-i+1)\theta \}).$$       

In the lexicographic ordering for words $u,v$ of length $m$ over $\{a<b\}$, $$u=u_{0}u_{1}\ldots u_{m-1} < v_{0}v_{1}\ldots v_{m-1} = v$$ if and only if for some index $i$, $u_{i}=a, v_{i}=b$, and for all $\ell<i, u_{\ell}=v_{\ell}$. 

\begin{theorem}\label{lexorder} Index in increasing order of magnitude the $m+1$ points \emph{0}, $\{-\theta \}, \{-2\theta\}, \{-3\theta \}, \ldots, \{-m\theta \}$ as $0 = q_{0} < q_{1} < q_{2}< \ldots < q_{m}$.  The $m+1$ factors of length $m$ of $s_{\theta}$ in lexicographic order correspond to the order of the $m+1$ half-open subintervals, $[{0}, q_{1}), [q_{1}, q_{2}), \dots, [q_{m-1}, q_{m}), [q_{m}, 1)$. 
\end{theorem}
\begin{proof}
 For $1 \leq g \leq m$, let $w_{g} = r'_0 \dots r'_{m-1}$ be a factor associated with the $g+1$-th subinterval $[q_g, q_{g+1})$, and similarly, $w_{g-1}=r''_0 \dots r''_{m-1}$ be a factor associated with the $g$-th subinterval $[q_{g-1}, q_{g})$, understanding $q_{m+1} = 1$.  That is, there exist factors of $s_\theta$ such that 
$$r'_0 \dots r'_{m-1} = s_{n'}\dots s_{n'+m-1},$$ 
and 
$$r''_{0}\dots r''_{m-1} = s_{n''} \dots s_{n''+m-1}.$$

We show that $w_{g-1}$ lexicographically precedes $w_{g}$, for $1 \leq g \leq m$. Then $R^{n'}(0)$ lies in the $g+1$-th interval of intersection, $[q_{g}, q_{g+1})$, i.e., $0 \in R^{-n'}([q_{g}, q_{g+1}))$, and similarly, $0 \in R^{-n''}([q_{g-1}, q_{g}))$. There exists some $h$, $0 \leq h < m$, such that $\{-h\theta\} = q_g = R^{-h}(0)$. This point is the included left endpoint of the $g+1$-th subinterval of intersection and the excluded right endpoint of the  $g$-th subinterval of intersection. We claim that $R^{-n'-h}(0) \in I_{b}$, $R^{-n'-h-1}(0) \in I_{a}$, while $R^{-n''-h}(0) \in I_{a}$, $R^{-n''-h-1}(0) \in I_{b}$.
That is, the $h$-th and $h+1$-th letters of $w_{g-1}$ and $w_{g}$ are $ab$ and $ba$, respectively, unless $h=m$, in which case the last letter of $w_{g-1}$ is $a$ and the last letter of $w_{g}$ is $b$.   

The included left endpoint of the interval $R^{-n'}([q_{g}, q_{g+1}))$ enclosing 0, is, by the calculation above,
 $$R^{-(n'+h)}(0) = R^{-n'}(R^{-h}(0)) = R^{-n'}(\{-h\theta\}) = R^{-n'}(q_{g}).$$ Therefore, since $q_{g} \ne 0$ and $R^{n'}(q_{g}) \geq 1-\theta$, $R^{n'}(q_{g}) \in I_{b}$, i.e., $R^{-(n'+h)}(0) \in I_{b}$, and the $h$-th letter of $w_{g}$ is $b$. $R^{-(n'+h+1)}(0)$ is therefore in the complementary interval, $I_{a}$, and the $h+1$-th letter of $w_g$ is $a$. If $h=m-1$, the first part holds: the last letter is $b$.

Similarly, $$R^{-(n''+h)}(0) = R^{-n''}(R^{-h}(0)) = R^{-n''}(\{-h\theta\}) = R^{-n''}(q_{g}),$$ which is the excluded right endpoint of the interval $R^{-n''}([q_{g-1}, q_{g}))$ enclosing 0, by the calculation above. Therefore, since $0 < R^{-n''}(q_{g}) < 1-\theta$, $R^{-n''}(q_{g}) \in I_{a}$, i.e., $ R^{-(n''+h)}(0) \in I_{a}$, and the $h$-th letter of $w_{g-1}$ is $a$. $R^{-(n''+h+1)}(0)$ is therefore in the complementary interval, $I_{b}$, and the $h+1$-th letter of $w_{g-1}$ is $b$. If $h=m-1$, again, the first part holds: the last letter is $a$.

All other corresponding letters of the two factors must be the same, or they would determine intersections in the complement of $[q_{g-1}, q_{g+1})$.  We therefore have the lexicographic ordering $w_{g-1} < w_{g}$, where $1 \leq g \leq m$. That is, the lexicographic ordering of the $m+1$ factors of length $m$, $w_0 < w_1 < \dots < w_{m}$, matches the increasing order of the points on the unit interval, $0 < q_1 < \dots < q_{m}$. Adjacent factors in the order differ by a single transposition, $ab \rightarrow ba$.
\end{proof}   

\subsection{Varieties and Frequencies for Partitioned Factors}

The lexicographic order for factors induces an order on the partitioned factors. We have the partitioned factors of $\mathcal{L}_{(m,k)}$ factored into $k$ components determined by the composition of $m$, $p_{1}+\dots+p_{k}=m$, and each factor $f$ of length $m$ is considered as a product of concatenated words $f_{i}$ of lengths $p_{i}$. As defined earlier, two factors $w$ and $v$, considered as partitioned factors, are equivalent if and only if their decompositions into constituent words $w_{i}$ and $v_{i}$ of lengths $p_{i}$ have the same heights, for all $i$ from 1 to $k$. The varieties of partitioned factors are the equivalence classes. That there are $k+1$ varieties already follows from Theorem~\ref{card}.  

Let the $k+1$ elements $q'_{0}=0<q'_{1}<\dots<q'_{k}$ be the fractional parts $\{0\theta \}, \{-p_{1}\theta \}, \{-(p_{1}+p_{2})\theta \}, \dots, \{-(p_{1}+ \dots +p_{k})\theta \} = \{-m\theta \}$, arranged in order of magnitude. Since $k \leq m$ and the $p_{j}$ determine a composition of $m$, the $q'_{j}$, ($ 0 \leq j \leq k$) are a subset of the $q_{i}$ ($0 \leq i \leq m$), defined in Theorem \ref{lexorder}.   
 
\begin{cor}\label{partfactors} Consider the partition of the unit circle given by the $k+1$ disjoint connected subintervals $[0, q'_{1}), [q'_{1}, q'_{2}), \dots, [q'_{k-1}, q'_{k}), [q'_{k}, 1)$. These subintervals define the $k+1$ varieties of partitioned factors, in the inherited lexicographic order, and their lengths are the frequencies of the respective partitioned factors.
\end{cor} 
\begin{proof} 

From the proof of Theorem~\ref{card}, we know that a variety $\lambda_{j}$ of partitioned factors in $\mathcal{L}_{(m,k)}$ gathers together the underlying factors $L$ that are adjacent in the lexicographic order. From Theorems~\ref{fact} and \ref{lexorder} we know that the underlying factors $u_{i}$ in lexicographic order correspond to the subintervals $[q_{i}, q_{i+1})$, $0 \leq i \leq m$, with the last of the $m+1$ factors corresponding to the subinterval $[q_{m}, 1)$. We also know, moreover, that the frequencies of the factors are given by the lengths of the subintervals. Since the points $q'_{j}$ are a subset of the $q_{i}$, for a given variety $\lambda_{j}$, we have $q'_{j} = q_{i}$, for some $i$, and $q'_{j+1} = q_{i+r}$, for some $1 \leq r \leq m-i$ (or in case $j=k$, $q'_{k+1} = 1$). Then by Corollary \ref{parfacs}, we have $Fr(\lambda_{j}) = \sum_{u \in L} Fr(u)$; i.e., the length of the subinterval $[q'_{j}, q'_{j+1})$. 
\end{proof}    

While frequencies of factors of length $m$ take on at most three distinct values (\cite{Berthe}), corresponding to the trivial composition of $m$ into $m$ parts, Corollary \ref{partfactors} shows that the number of values the frequencies of partitioned factors can take on is limited only by the complexity of the composition of $m$ into $k$ parts. The frequencies can be any of the possible sums of lengths of adjacent subintervals $[q_{i}, q_{i+1})$ of Theorem \ref{lexorder}, depending on the composition.

The theorems and corollaries of the second part of the paper are exemplified in Figure~\ref{fig: Freq}, in the context of the Sturmian word from Example~\ref{l23}.

\begin{figure}
\center  
\includegraphics [width=12.6 cm] {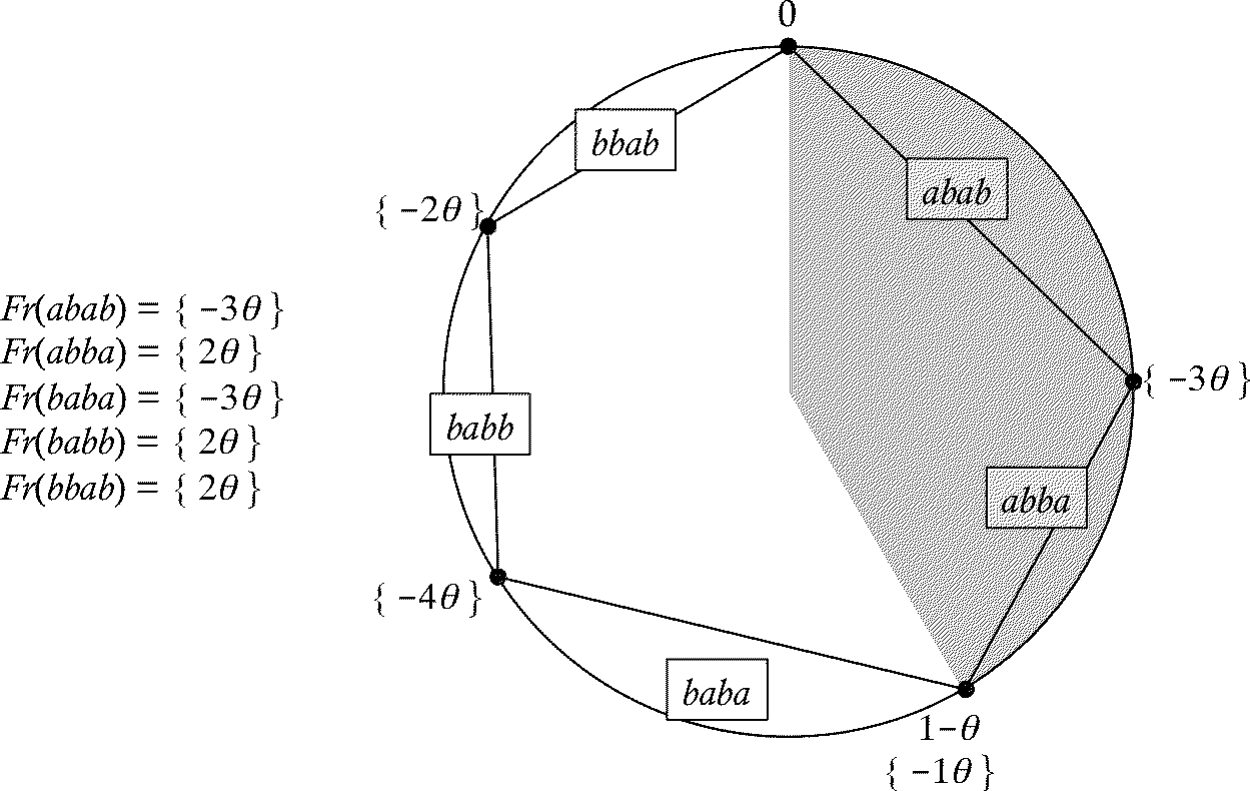}
\includegraphics [width=12.6 cm]{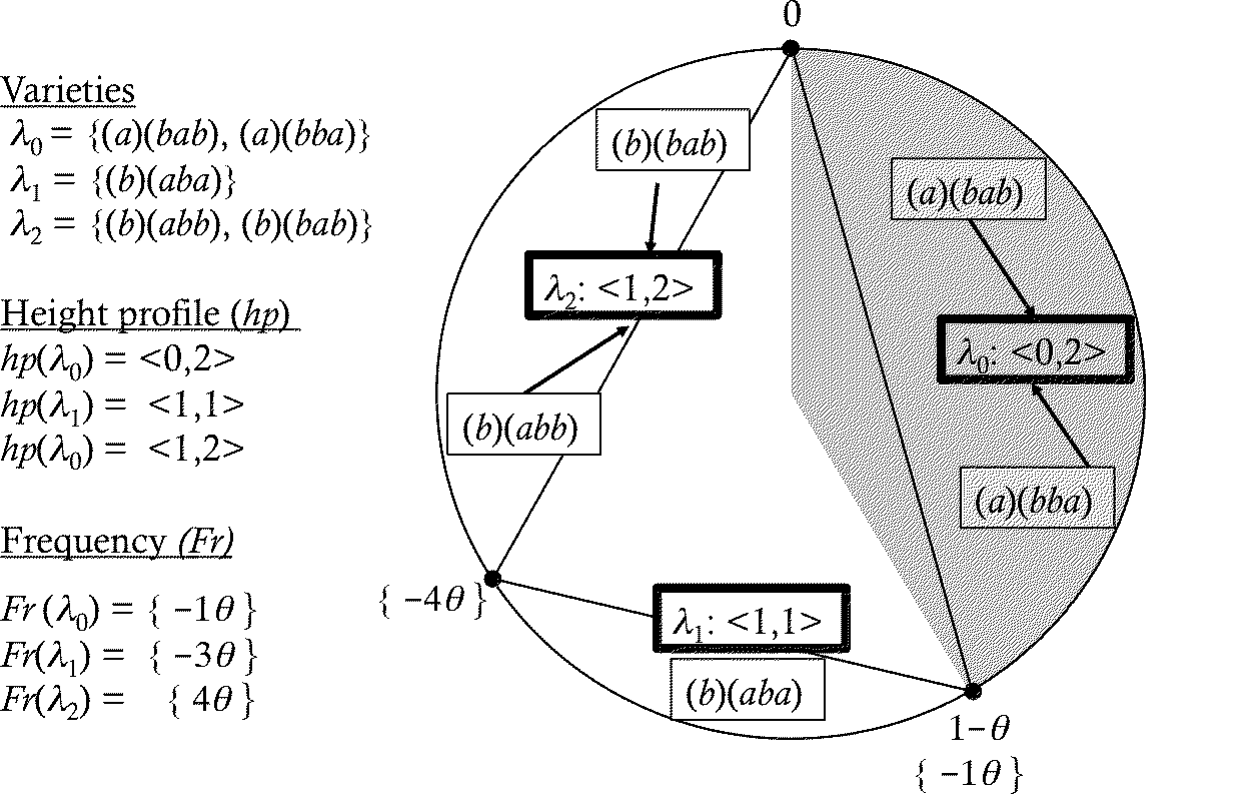}   
\caption{With $\theta = \log_{2}(3/2)$, frequencies of factors and partitioned factors of the Sturmian word $s_{\theta} = abababbababbabababbababbabababbababb\dots$. (See Example \ref{l23}.) The shaded areas in both circles represent the half-open interval, $[0, 1-\theta)$, where $\{n\theta\} \longrightarrow a$. The non-shaded areas repesent the half-open interval $[1-\theta, 1)$, $\{n\theta\}\longrightarrow b$. Upper diagram: the frequencies of factors $\mathcal{L}_{4}$. Lower diagram: the frequencies of partitioned factors $\mathcal{L}_{(4,2)}$, under the composition (1,3).}
\label{fig: Freq}
\end{figure}

 \section{Coda}
The motivation for this study is two-fold: the first is in the recognition of the fact that the diatonic scale may be represented by a conjugate of the Christoffel word $aaabaab$, where $a$ represents a whole tone or major second and $b$ represents a semitone or minor second; the second stems from two papers in mathematical music theory, \cite{CloughMyerson85, CloughMyerson86}. (See Figure \ref{fig: Partitioned_Christoffel} and Example \ref{seven}).
Clough and Myerson show that the ordinary diatonic scale in equal temperament exhibits ``Myhill's Property,'' which musicians will recognize in the fact that  diatonic intervals---seconds, thirds, fourths, fifths, etc.---come in two varieties. Musicians refer to these varieties as intervallic ``qualities,'' so, for example, we have \emph{major and minor} seconds, \emph{major and minor} thirds, \emph{perfect and augmented} fourths, \emph{perfect and diminished} fifths, and so on. There are, then, exactly two specific qualities for each span for intervals up to but not including the octave. This property, they showed, holds in turn for collections of ``chords,'' i.e., larger collections of intervals. Chords with $n$ notes come in $n$ varieties. For example,  three-note diatonic ``triads'' come in three varieties, known as major, minor, and diminished: all three exist in any given diatonic scale. Finally, using an ordering of the  scale known as the ``circle of fiths,'' they showed how to calculate the multiplicities of each type.  Proposition \ref{mp} showed that Myhill's Property characterizes Christoffel words. In Theorem \ref{card} and Corollary \ref{mult}, we presented Clough and Myerson's results within the context of Christoffel  words, while Section \ref{pfsturm} extended these results to Sturmian words. 

\bibliographystyle{elsarticle-harv}
\bibliography{parrefs}
\end{document}